\documentclass[]{gAPA2e}
\usepackage{color}

\begin{document}

%

\title{Unique continuation and approximate controllability\\for a degenerate parabolic equation}

\author{P. Cannarsa$^{\rm a}$$^{\ast}$\thanks{$^\ast$Corresponding author. Email: cannarsa@mat.uniroma2.it
\vspace{6pt}} 
J. Tort$^{\rm b}$ and M. Yamamoto$^{\rm c}$\\\vspace{6pt}  $^{\rm a}${\em{Dipartimento di Matematica, Universit\`a di Roma ``Tor Vergata'',
Via della Ricerca Scientifica 1, 00133 Roma, Italy}}; 
$^{\rm b}${\em{Institut de Math\'ematiques de Toulouse, U.M.R. C.N.R.S. 5219, Universit\'e Paul Sabatier Toulouse III, 118 route de Narbonne, 31 062 Toulouse Cedex 4, France}};  $^{\rm c}${\em{Department of Mathematical Sciences, The University of Tokyo, Komaba Meguro Tokyo 153-8914, Japan}}\\\vspace{6pt}\received{v3.3 released May 2008} }

\maketitle

\begin{abstract}
This paper studies unique continuation for weakly degenerate parabolic equations in one space dimension. A new Carleman estimate of local type is obtained to deduce that all solutions that vanish on the degeneracy set, together with their conormal derivative, are identically equal to zero. An approximate controllability result for weakly degenerate parabolic equations under Dirichlet boundary condition is deduced.\bigskip

\begin{keywords}degenerate parabolic equations; unique continuation;
approximate controllability; local Carleman estimate
\end{keywords}
\begin{classcode} 35K65; 93B05; 35A23; 93C20
\end{classcode}\bigskip

\end{abstract}

\section{Introduction}
We consider a parabolic equation degenerating at the boundary 
of the space, which is related to a motivating example of 
a Crocco-type equation coming from
the study of the velocity field of a laminar flow on a 
flat plate (see, e.g., \cite{Bu}).

The null controllability of degenerate parabolic operators in one space dimension has been well studied for locally distributed controls. For instance, in \cite{Cannarsa, Cannarsa1}, the problem 
\\
\begin{equation*}
\left\{ 
\begin{array}{lc}
u_{t}-(x^\alpha u_{x})_{x}=\chi_\omega h &\quad \left( t,x\right) \in Q:=(0,1)\times(0,T)   \\ 
u\left(1, t\right)=0 &\quad t\in \left(0,T\right) \\ 
\mbox{and }\left\{ 
\begin{array}{ll}
u\left( 0,t\right) =0 & \quad \mbox{for }0\leq \alpha <1 \\ 
(x^\alpha u_{x})\left( 0,t\right) =0 & \quad \mbox{for }1\leq \alpha <2
\end{array}
\right.
& \quad t\in \left( 0,T\right) \\ 
u\left(x, 0\right) =u_{0}(x) & \quad x\in \left( 0,1\right)\,,
\end{array}
\right.  
\end{equation*}
\\
where $\chi_\omega$ denotes the characteristic function of $\omega=(a,b)$ with $0<a<b<1$, is shown to be null controllable in $L^2(0,1)$ in any time $T>0$. Generalizations of the above result to semilinear problems and nondivergence form operators can be found in \cite{Alabau} and \cite{non-div0, non-div}, respectively. The global Carleman estimate derived in \cite{Cannarsa} was also used in \cite{Cannarsa5} to prove Lipschitz stability estimates for inverse problems relative to degenerate parabolic operators.

It is a commonly accepted viewpoint that, if a system is controllable via locally distributed controls, then it is also controllable via boundary controls and vice versa. This is indeed the case for uniformly parabolic operators. 
For degenerate operators, on the contrary, no null controllability result  is available in the literature---to our best knowledge---when controls act on `degenerate' parts of the boundary. Indeed, in this case,
switching from locally distributed to boundary controls is by no means automatic for at least two reasons. 
In the first place,  Dirichlet boundary data can only be imposed in weakly degenerate settings (that is, when $0\le\alpha<1$), since otherwise solutions may not define a trace on the boundary, see \cite[section 5]{CRV}. Secondly, the standard technique which consists in enlarging the space domain and placing an `artificial' locally distributed control in the enlarged region, would lead to an unsolved problem in the degenerate case. Indeed, such a procedure requires  being able to solve the null controllability problem for an operator which degenerates in the interior of the space domain, with controls acting only on one side of the domain with respect to the point of degeneracy.

In this paper, we establish a simpler result, that is,
the approximate controllability via controls at
the `degenerate' boundary point for the weakly degenerate parabolic 
operator
$$
Pu:=u_t-(x^\alpha u_x)_x\quad\text{in}\quad Q \qquad(0\le \alpha<1).
$$ 
In order to achieve this, we follow the classical duality argument that reduces the problem to the 
unique continuation for the adjoint of $P$, that is, the operator
\begin{equation*}
Lu:=u_t+(x^\alpha u_x)_x\quad\text{in}\quad Q
\end{equation*}
with
 boundary conditions
\begin{equation}\label{intro:bc}
u(0,t)=(x^\alpha u_x)(0,t)=0\,.
\end{equation}
To solve such a problem, in section 2 of this paper
we derive new local Carleman estimates for  $L$, 
in which the weight function exhibits a decreasing 
behaviour with respect to $x$ (Theorem~\ref{theo1}).
Then, in section 3, we obtain our unique continuation result proving that any solution $u$ of $Lu=0$ in $Q$, which satisfies \eqref{intro:bc}, must vanish identically in $Q$ (Theorem 3.1). 
Finally, in section 4, we show how to deduce the approximate 
controllability with Dirichlet boundary control for the weakly degenerate problem ($0\leq \alpha <1$)\\
\begin{equation*}
\left\{ 
\begin{array}{lc}
u_{t}-(x^\alpha u_{x})_{x}=0 &\quad \left( x,t\right) \in Q \\ 
u\left( 1,t\right)=0 &\quad t\in \left(0,T\right) \\ 
u\left( 0,t\right) =g(t) & \quad t\in \left( 0,T\right) \\ 
u\left( x,0\right) =u_{0}(x) & \quad x\in \left( 0,1\right)\,.
\end{array}
\right.  
\end{equation*}
The outline of this paper is the following. In section~\ref{se:Carle}, we derive our local Carleman estimate. Then, in section~\ref{se:uc}, we apply such an estimate to deduce a unique continuation result for $L$.  Finally, in section~\ref{se:ac}, we obtain approximate controllability for $P$ with Dirichlet boundary controls as a consequence of unique continuation.
\section{A Carleman estimate with decreasing-in-space weight functions}
\label{se:Carle}
We begin by recalling the definition of the  function spaces that will be used throughout this paper. The reader is referred to \cite{Alabau,Cannarsa} for more details on these spaces.

For any $\alpha\in(0,1)$ we  define $H_{\alpha}^{1}\left( 0,1\right)$ to be the space of all absolutely continuous functions $u:[0,1]\to\mathbb R$ such that
\begin{equation*}
\int_0^1x^{\alpha}|u_{x}(x)|^2\,dx <\infty 
\end{equation*}
where $u_x$ denotes the derivative of $u$. 
Like the analogous property of standard Sobolev spaces, one can prove that $H_{\alpha}^{1}\left( 0,1\right)\subset C([0,1])$. So, one can also set
\[
H_{\alpha,0}^{1}\left( 0,1\right)=\left\{ u\in H_{\alpha}^{1}\left( 0,1\right): u\left( 0\right) =u\left( 1\right) =0\right\}.
\]
Now, define the  operator $A:D(A)\subset L^{2}\left( 0,1 \right)\to L^{2}\left( 0,1 \right)$ by
\[
\begin{cases}
D(A):=\left\{ u\in H_{\alpha,0}^{1}\left( 0,1\right)~:~x^{\alpha}u_{x}\in H^{1}\left(
0,1\right) \right\} \\
Au=\left(x^\alpha u_{x}\right) _{x}\,, \quad \forall u \in D(A)\,.
\end{cases}
\]
We recall that $A$ is the infinitesimal generator of an analytic semigroup of contractions on $L^{2}\left( 0,1 \right)$, and $D(A)$ is a Banach space with the graph norm
\begin{equation*}
|u|_{D(A)}=\|u\|_{L^{2}\left( 0,1 \right)}+\|Au\|_{L^{2}\left(0,1 \right)}\,.
\end{equation*}

\begin{example}
As one can easily check by a direct calculation,  $f(x)= 1-x^{1-\alpha}$ belongs to $H_{\alpha}^{1}\left( 0,1\right)$ and
\begin{equation*}
(x^\alpha u_x)_x=0\qquad \forall x \in [0,1]\,.
\end{equation*}
However, $f\notin D(A)$ since $f(0)=1$.
\label{lem1}
\end{example}

\begin{lemma}\label{le:DA}
 Let $u\in D(A)$ be such that $x^\alpha u_x\to 0$ as $x\to 0$. Then
 \begin{equation}\label{eq:DA1}
|x^\alpha u_x(x)|\le |u|_{D(A)}\sqrt{x}\qquad\forall\, x\in [0,1]
\end{equation}
and
 \begin{equation}\label{eq:DA0}
|x^{\alpha-1} u(x)|\le \frac 2{3-2\alpha}\, |u|_{D(A)}\sqrt{x}\qquad\forall\, x\in [0,1]
\end{equation}
Moreover, for any $\beta>0$ 
there is a constant $c(\beta)$ such that
\begin{equation}\label{eq:DA2}
\int_0^1x^{2\alpha+\beta-4}u^2\,dx\;+\;\int_0^1 x^{2\alpha+\beta-2}u^2_x\,dx\;\le\; c(\beta)\,|u|_{D(A)}^2\,.
\end{equation}
\begin{proof}\rm
Let $u\in D(A)$ be such that $x^\alpha u_x\to 0$ as $x\to 0$.  Since
 \begin{equation*}
x^\alpha u_x(x)=\int_0^x\frac{d}{ds}\,\Big(s^\alpha \frac{du}{ds}(s)\Big)ds\,,
\end{equation*}
\eqref{eq:DA1} follows by H\"older's inequality. Then, owing to \eqref{eq:DA1},
\begin{equation*}
|u(x)|\le \int_0^x\Big| s^\alpha  \frac{du}{ds}(s)\Big|\,s^{-\alpha}\,ds
\le |u|_{D(A)} \int_0^x s^{\frac 12-\alpha}\,ds
\end{equation*}
which in turn yields \eqref{eq:DA0}.
Next, in view of \eqref{eq:DA1},
\begin{equation*}
\int_0^1 x^{2\alpha+\beta-2}u^2_x\,dx\;\le |u|^2_{D(A)}\;\int_0^1 x^{\beta-1}\,dx = \frac 1\beta\;|u|^2_{D(A)}\,.
\end{equation*}
Finally, on account of \eqref{eq:DA0},
\begin{equation*}
\int_0^1x^{2\alpha+\beta-4}u^2\,dx \le \Big(\frac 2{3-2\alpha}\Big)^2\, |u|_{D(A)}^2\;\int_0^1 x^{\beta-1}\,dx\,.
\end{equation*}
The proof of \eqref{eq:DA2} is thus complete.
\end{proof}
\end{lemma}
\subsection{Statement of the Carleman estimate}
Let $T>0$. Hereafter, we set
\begin{equation*}
Q=(0,1) \times (0,T)\,.
\end{equation*}
Moreover, for any integrable function $f$ on $Q$, we will use the abbreviated notation
\begin{equation*}
\int_Qf=\int_Qf(x,t)dxdt\,.
\end{equation*}

Let $0<\alpha<1$ and fix $\beta \in (1-\alpha, 1-\frac{\alpha}{2})$. Define weight functions $l,p$ and $\phi$ as 
\begin{equation}
\forall t \in (0,T), \quad l(t):=\frac{1}{t(T-t)},
\label{weightl}
\end{equation}
\begin{equation}
\forall x \in (0,1), \quad p(x):=-x^\beta
\label{weightbeta}
\end{equation}
\noindent and
\begin{equation}
\forall (x,t) \in Q, \quad \phi(x,t):=p(x)l(t).
\label{weightphi}
\end{equation}

\bigskip

\noindent For any function $v \in L^2(0,T;D(A))\cap H^1(0,T;L^2(0,1))$, we set 
$$Lv:=v_t+(x^\alpha v_x)_x\,.$$ 

We will prove the following Carleman estimate:
\begin{theorem}
Let $v \in L^2(0,T;D(A))\cap H^1(0,T;L^2(0,1))$ and suppose that, for a.e. $t \in (0,T)$,
\[
v(0,t)=(x^\alpha v_x)(0,t)=v(1,t)=(x^\alpha v_x)(1,t)=0\,.
\]
\noindent Then, there exist constants $C=C(T,\alpha,\beta)>0$ and $s_0=s_0(T,\alpha,\beta)>0$ such that, for all $s\geq s_0$,
\begin{equation}
\int_Q [s^3l^3x^{2\alpha+3\beta-4}+slx^{2\alpha+\beta-4}]v^2e^{2s\phi}
+ \int_Q slx^{2\alpha+\beta-2}v_x^{2}e^{2s\phi}
\leq C \int_Q |Lv|^2e^{2s\phi}.
\label{Carleman}
\end{equation}
\label{theo1}
\end{theorem}

The proof is inspired by \cite{FuIm} and \cite{Im},  where global Carleman estimates
for uniformly parabolic equations were first obtained, and by \cite{Alabau}, \cite{Cannarsa}, and \cite{Cannarsa5}, where this technique was adapted to degenerate parabolic operators by the choice of appropriate weight functions.

We now proceed to derive another Carleman  estimate which  follows  from
\eqref{Carleman} and yields  unique continuation, deferring the proof of Theorem \ref{theo1}  to the next section. 
\begin{corollary}
Let $v \in L^2(0,T;D(A))\cap H^1(0,T;L^2(0,1))$ and suppose that, for a.e. $t \in (0,T)$,
\[
v(0,t)=(x^\alpha v_x)(0,t)=v(1,t)=(x^\alpha v_x)(1,t)=0\,.
\]
\noindent Then there exist constants $C=C(T,\alpha,\beta)>0$ and $s_0=s_0(T,\alpha,\beta)>0$ such that, for all $s\geq s_0$,
\begin{equation}
\int_Q s^3l^3v^2e^{2s\phi}+\int_Q slx^{2\alpha+\beta-2}v_x^{2}e^{2s\phi} \leq 
C \int_Q |Lv|^2e^{2s\phi}\,.
\label{Carleman2}
\end{equation}
\label{cor1}
\end{corollary}
\begin{proof}
 Since $\beta <1-\frac{\alpha}{2}$, we have that $4\beta<4-2\alpha$ and $2\alpha+4\beta-4<0$. Moreover,  $2\alpha+3\beta-4 < 2\alpha+4\beta-4<0$ since $\beta >0$. Consequently,  $x^{2\alpha+3\beta-4}\geq 1$ for all $x\in (0,1)$. Then,
\[
\int_Q s^3l^3v^2e^{2s\phi} \leq \int_Q s^3l^3x^{2\alpha+3\beta-4}v^2e^{2s\phi}
\]
\noindent and the proof  is  complete.
\end{proof}

\subsection{Proof of Theorem \ref{theo1}}
Let $v \in L^2(0,T;D(A))\cap H^1(0,T;L^2(0,1))$ and suppose  that, for a.e. $t \in (0,T)$,
\begin{equation}
v(0,t)=(x^\alpha v_x)(0,t)=v(1,t)=(x^\alpha v_x)(1,t)=0.
\label{boundaryv}
\end{equation}
\begin{lemma}
 Let $w:=ve^{s\phi}$. Then $w$ belongs to $L^2(0,T;D(A))\cap H^1(0,T;L^2(0,1))$ and satisfies,  for a.e. $t \in (0,T)$,
\begin{equation}
w(0,t)=w(1,t)=0
\label{boundaryw}
\end{equation}
and
\begin{equation}
(x^\alpha w_x)(0,t)=(x^\alpha w_x)(1,t)=0\,.
\label{boundaryw1}
\end{equation}
Moreover,  $w$ satisfies $L_s w=e^{s\phi}Lv$, where $L_sw=L_s^+w+L_s^-w$, and
\begin{equation}\label{eq:Lpm}
\begin{split}
&L_s^+w=(x^\alpha w_x)_x-s\phi_t w+s^2x^\alpha \phi_x^2w\\
&L_s^-w=w_t-2sx^\alpha\phi_xw_x-s(x^\alpha \phi_x)_xw\,.
\end{split}
\end{equation}
Furthermore, $L_s^+w, L_s^-w\in L^2(Q)$ and
\begin{multline}
\int_QL_s^+w\, L_s^-w = \frac{s}{2}\int_Q \phi_{tt}w^2+s\int_Q x^\alpha(x^\alpha\phi_x)_{xx}ww_x+2s^2\int_Q x^\alpha \phi_x \phi_{tx}w^2 \\
+ s \int _Q (2x^{2\alpha}\phi_{xx}+\alpha x^{2\alpha-1}\phi_x)w_x^2+s^3\int_Q(2x^\alpha \phi_{xx}+\alpha x^{\alpha-1}\phi_x)x^\alpha \phi_x^2w^2\,.
\label{lemACG}
\end{multline}
\end{lemma}
\begin{proof}
One easily checks that, for a.e. $t \in (0,T)$,
\[
x^\alpha w_x=s x^{\alpha}\phi_xve^{s\phi}+x^\alpha v_xe^{s\phi}\,.
\]
Note that, because of our choice \eqref{weightbeta}, $\phi_x=-\beta l x^{\beta-1}$, so that $x^\alpha \phi_x=- \beta l x^{\alpha+\beta-1}$. Then, the fact that $w\in L^2(0,T;D(A))\cap H^1(0,T;L^2(0,1))$, as well as \eqref{boundaryw} and \eqref{boundaryw1}, follows from Lemma~\ref{le:DA} and \eqref{boundaryv}. Similarly, one can show $L_s^+w, L_s^-w\in L^2(Q)$. As for  \eqref{lemACG}, integrating by parts as in \cite[Lemma 3.4]{Alabau} one obtains
\begin{multline*}
\int_QL_s^+w\, L_s^-w = \frac{s}{2}\int_Q \phi_{tt}w^2+s\int_Q x^\alpha(x^\alpha\phi_x)_{xx}ww_x+2s^2\int_Q x^\alpha \phi_x \phi_{tx}w^2 \\
+ s \int _Q (2x^{2\alpha}\phi_{xx}+\alpha x^{2\alpha-1}\phi_x)w_x^2+s^3\int_Q(2x^\alpha \phi_{xx}+\alpha x^{\alpha-1}\phi_x)x^\alpha \phi_x^2w^2
\\
+\int_0^T\Big[x^\alpha w_xw_t-s\phi_x(x^\alpha w_x)^2+s^2x^\alpha \phi_t\phi_xw^2-s^3x^{2\alpha}\phi_x^3w^2-sx^\alpha (x^\alpha \phi_x)_xww_x\Big]_{x=0}^{x=1}\,dt
\end{multline*} 
Recalling Lemma~\ref{le:DA} once again, and the boundary conditions \eqref{boundaryw} and \eqref{boundaryw1}, it is easy to see that the boundary terms  vanish in the above identity, which therefore reduces to \eqref{lemACG}.
\end{proof}

We can now proceed with the proof of  Theorem~\ref{theo1}. Since  $e^{s\phi}Lv=L_s^+w+L_s^-w$, identity \eqref{lemACG} yields
\begin{multline*}
\|e^{s\phi}Lv\|_{L^2(Q)}^2 \geq \frac{s}{2}\int_Q \phi_{tt}w^2+s\int_Q x^\alpha(x^\alpha\phi_x)_{xx}ww_x+2s^2\int_Q x^\alpha \phi_x \phi_{tx}w^2 \\
+ s \int _Q (2x^{2\alpha}\phi_{xx}+\alpha x^{2\alpha-1}\phi_x)w_x^2+s^3\int_Q(2x^\alpha \phi_{xx}+\alpha x^{\alpha-1}\phi_x)x^\alpha \phi_x^2w^2.
\end{multline*}
 Let us denote by ${\sum_{k=1}^{5}J_k}$ the right-hand side of the above estimate. We will now  use the properties  of the weight functions in \eqref{weightl}, \eqref{weightbeta} and \eqref{weightphi} to bound each  $J_k$. 

First of all, we have
\[
|J_1 |
=\Big|\frac{s}{2}\int_Q \phi_{tt}w^2\Big| \leq \frac{s}{2} \int_Q |
l^{\prime \prime}|w^2.
\]
\noindent Yet, one can easily check that there exists a constant $C=C(T)>0$ such that, for all $t\in (0,T)$, $|l^{\prime \prime}(t)|\leq C l^3(t)$. Then, there exists $C=C(T)>0$ such that
\begin{equation}
|J_1| \leq C s\int_Q l^3 w^2.
\label{estimJ1}
\end{equation}

Now, to estimate $J_2$ observe that, in view of \eqref{boundaryw}, we have
\[
J_2=\frac{s}{2}\int_Q x^\alpha(x^\alpha\phi_x)_{xx}\partial _x(w^2)=-\frac{s}{2}\int_Q (x^\alpha(x^\alpha\phi_x)_{xx})_xw^2\,.
\]
\noindent Moreover, for all $(x,t) \in (0,1)\times (0,T)$, $\phi_x(x,t)=-\beta l(t)x^{\beta-1}$. Then, $x^\alpha \phi_x(x,t)=-\beta l(t)x^{\alpha+\beta-1}$.  Therefore, for all $(x,t) \in (0,1)\times (0,T)$,
\[
x^\alpha(x^\alpha\phi_x)_{xx}=-\beta (\alpha+\beta-1)(\alpha+\beta-2)l(t)x^{2\alpha+\beta-3}.
\]
\noindent Eventually,
\begin{equation*}
(x^\alpha(x^\alpha\phi_x)_{xx})_x=-\beta (\alpha+\beta-1)(\alpha+\beta-2)(2\alpha+\beta-3)l(t)x^{2\alpha+\beta-4}\,.
\end{equation*}

\noindent Let us now show  that the product $\beta (\alpha+\beta-1)(\alpha+\beta-2)(2\alpha+\beta-3)$ is positive. First of all, since $1-\alpha<\beta$, we have $\alpha+\beta-1>0$. Since $\alpha<1$ and $\beta <1$, $\alpha+\beta-2<0$. Moreover,
\[
2\alpha+\beta-3<2\alpha+1-\frac{\alpha}{2}-3=\frac{3}{2}\alpha-2<0,
\]
\noindent since $\alpha<1$. Therefore, $\beta (\alpha+\beta-1)(\alpha+\beta-2)(2\alpha+\beta-3)>0$. Then, there exists $C=C(\alpha,\beta)>0$ such that
\begin{equation}
J_2\geq C(\alpha,\beta) s \int_Q lx^{2\alpha + \beta -4}w^2.
\label{estimJ2}
\end{equation}

Next, observe that
\[
J_3=2s^2 \int_Q x^\alpha(-\beta x^{\beta-1}l(t))(-\beta x^{\beta-1} l^{\prime}(t))w^2=2s^2 \int_Q l(t)l^{\prime}(t)\beta^2 x^{\alpha+2\beta-2}w^2.
\]
\noindent Also, $|l(t)l^{\prime}(t)|\leq C l^3(t)$ for all $t\in (0,T)$ and some constant $C=C(T)>0$. Then, 
\begin{equation}
|J_3| \leq C s^2 \int_Q l^3(t)x^{\alpha+2\beta-2}w^2.
\label{estimJ3}
\end{equation}

 Computing the derivatives in $J_4$, one has
\begin{multline*}
J_4=s\int_Q (-2\beta (\beta-1)l(t)x^{2\alpha+\beta-2}-\alpha\beta l(t) x^{\alpha+\alpha-1+\beta-1})w_x^2
\\
=s\int_Q l(t)\beta x^{2\alpha+\beta-2}(-2\beta+2-\alpha)w_x^2\,.
\end{multline*}
\noindent Yet, $\beta<1-\frac{\alpha}{2}$, so that $-2\beta-\alpha+2>0$. Then, for some $C=C(\alpha,\beta)>0$
\begin{equation}
J_4=C(\alpha,\beta) s \int_Q l(t)x^{2\alpha+\beta-2} w_x^2.
\label{estimJ4}
\end{equation}

 Finally, arguing in the same way for $J_5$ we have
$$
J_5  =s^3\int_Q (-2\beta (\beta-1)l(t)x^{\alpha+\beta-2}-\alpha\beta l(t)x^{\alpha-1+\beta-1})l^2(t)\beta^2 x^{\alpha+2\beta-2}w^2
$$
$$
= s^3 \int_Q \beta^3 l^3(t)(-2\beta+2-\alpha)
x^{2\alpha+3\beta-4}w^2.
$$
\noindent Since $-2\beta+2-\alpha>0$, there exists $C=C(\alpha,\beta)>0$ such that
\begin{equation}
J_5=C(\alpha,\beta) s^3 \int_Q l^3(t)x^{2\alpha+3\beta-4}w^2.
\label{estimJ5}
\end{equation}

\noindent Coming back to \eqref{lemACG}, and using \eqref{estimJ1}, \eqref{estimJ2}, \eqref{estimJ3}, \eqref{estimJ4} and \eqref{estimJ5}, one has 

\begin{multline*}
\|e^{s\phi}Lv\|_{L^2(Q)}^2 \geq -C s\int_Q l^3 w^2+C(\alpha,\beta) s \int_Q lx^{2\alpha +\beta-4}w^2-C s^2 \int_Q l^3(t)x^{\alpha+2\beta-2}w^2 \\
+C(\alpha,\beta) s \int_Q l(t)x^{2\alpha+\beta-2} w_x^2 +C(\alpha,\beta) s^3 \int_Q l^3(t)x^{2\alpha+3\beta-4}w^2.
\end{multline*}

\noindent So, we can immediately deduce that, for some constant $C=C(T,\alpha,\beta)>0$,
\begin{multline}
\int_Q\Big(s^3l^3(t)x^{2\alpha+3\beta-4}+sl(t)x^{2\alpha+\beta-4}\Big)w^2+\int_Q sl(t)x^{2\alpha+\beta-2}w_x^2 \\
\leq C\left(\|e^{s\phi}Lv\|_{L^2(Q)}^2+s\int_Ql^3(t)w^2+s^2\int_Q l^3(t)x^{\alpha+2\beta-2}w^2\right).
\label{estim2}
\end{multline}

\noindent Now, we are going to absorb the two rightmost terms of \eqref{estim2} by the left-hand side. First of all, we note that
\[
2\alpha+3\beta-4-(\alpha+2\beta-2)=\alpha+\beta-2<0.
\]
\noindent As a consequence, since $0<x<1$,
\[
\int_Q l^3(t)x^{\alpha+2\beta-2}w^2\leq \int_Q l^3(t)x^{2\alpha+3\beta-4}w^2.
\]
\noindent Moreover, we have already mentioned that $2\alpha+3\beta-4<0$, so that for all $x \in (0,1)$, $1\leq x^{2\alpha+3\beta-4}$ and 
\[
\int_Ql^3(t)w^2 \leq \int_Q l^3(t)x^{2\alpha+3\beta-4}w^2.
\]
\noindent Then, \eqref{estim2} becomes
\begin{multline}
\int_Q(s^3l^3(t)x^{2\alpha+3\beta-4}+sl(t)x^{2\alpha+\beta-4})w^2+\int_Q sl(t)x^{2\alpha+\beta-2}w_x^2 \\
\leq C\left(\|e^{s\phi}Lv\|_{L^2(Q)}^2+(s+s^2)\int_Q l^3(t)x^{2\alpha+3\beta-4}w^2\right),
\label{estim3}
\end{multline}
\noindent with $C=C(T,\alpha,\beta)>0$. Now, there exists $s_0=s_0(T,\alpha,\beta)>0$ such that, for all $s\geq s_0$, $C(s+s^2)\leq s^3/2$. Therefore, for all $s\geq s_0$ and some $C=C(T,\alpha,\beta)>0$,
\begin{multline}
\int_Q(s^3l^3(t)x^{2\alpha+3\beta-4}+sl(t)x^{2\alpha+\beta-4})w^2+\int_Q sl(t)x^{2\alpha+\beta-2}w_x^2
\\
\leq C \|e^{s\phi}Lv\|_{L^2(Q)}^2.
\label{estim4}
\end{multline}
\noindent Eventually,  recalling that $w=ve^{s\phi}$, we have
\begin{multline}
\int_Q(s^3l^3(t)x^{2\alpha+3\beta-4}+sl(t)x^{2\alpha+\beta-4})v^2e^{2s\phi}+\int_Q sl(t)x^{2\alpha+\beta-2}w_x^2
\\
\leq C \|e^{s\phi}Lv\|_{L^2(Q)}^2.
\label{estim5}
\end{multline}
\noindent Moreover, $v_xe^{s\phi}=w_x-s\phi_xve^{s\phi}$. Therefore,
\[
\int_Q sl(t)x^{2\alpha+\beta-2}v_x^2e^{2s\phi} \leq 2 \int_Q sl(t)x^{2\alpha+\beta-2}w_x^2+ 2s^3\beta^2 \int_Q l^3x^{2\beta-2+2\alpha+\beta-2}v^2e^{2s\phi}\,.
\]
\noindent Thus,
\[
\int_Q sl(t)x^{2\alpha+\beta-2}v_x^2e^{2s\phi} \leq 2 \int_Q sl(t)x^{2\alpha+\beta-2}w_x^2+ 2s^3\beta^2 \int_Q l^3x^{2\alpha+3\beta-4}v^2e^{2s\phi}\,.
\]
\noindent The proof of Theorem \ref{theo1} is then completed thanks to \eqref{estim5}. 

\section{A unique continuation result}
\label{se:uc}
In this section, our goal is to show the following unique continuation property for the `adjoint operator'
\begin{equation*}
Lv=v_t+(x^\alpha v_x)_x\quad\text{in}\quad Q\,.
\end{equation*}
\begin{theorem}
Let $v \in L^2(0,T;D(A)) \cap H^1(0,T,L^2(0,1))$ and suppose that, for a.e. $t \in (0,T)$,
\begin{equation}
v(0,t)=(x^\alpha v_x)(0,t)=0.
\label{hypUC}
\end{equation}
\noindent If $Lv\equiv 0$ in $Q$, then $v\equiv 0$ in $Q$.
\label{theo2}
\end{theorem}
\begin{proof}
Let $0<\delta<1$ and $\Omega_\delta:=\{x\in (0,1): p(x)>-\delta\}$. The first step of the proof consists in proving that $v\equiv 0$ in $\Omega_\delta \times (\frac{T}{4},\frac{3T}{4})$. First of all, let us note that
\begin{equation}
x\in \Omega_\delta \text{ if and only if } x<\delta^{1/\beta}.
\label{omegadelta}
\end{equation}

\noindent Now, let us take $ \eta\in(\delta,1)$ and $\chi \in C^\infty(\mathbb R)$ such that $0\leq \chi\leq 1$ and 
\[
\chi(x)= \left\{
\begin{array}{ll}
1 & \quad x \in \Omega_\delta \\
0 & \quad x \notin \Omega_{\eta}
\end{array}
\right . .
\]
\noindent From the definition of $\chi$ above and  \eqref{omegadelta}, we deduce that
\begin{equation}
\forall x \in [0,\delta^{1/\beta}], \quad \chi(x)=1,
\label{chi0}
\end{equation}
\noindent and
\begin{equation}
\forall x \in [\eta^{1/\beta},1], \quad \chi(x)=0.
\label{chi1}
\end{equation}
\noindent Define $u \in L^2(0,T;D(A)) \cap H^1(0,T;L^2(0,1))$ by $u:=\chi v$, and observe that
\[
Lu=\partial_t u+(x^\alpha u_x)_x=\chi v_t+(x^\alpha(\chi v)_x)_x.
\]
\noindent Hence, after some standard computations, we get
\begin{equation}
Lu=\chi^{\prime\prime}x^\alpha v+\chi^\prime \alpha x ^{\alpha-1} v +2\chi^\prime x ^{\alpha} v_x.
\label{Lu}
\end{equation}

\noindent 
In order to appeal to Corollary~\ref{cor1}, we have to check that $u$ satisfies the required boundary conditions. First of all, for a.e. $t \in (0,T)$, $u(0,t)=\chi(0)v(0,t)=0$ by \eqref{hypUC}, and $u(1,t)=\chi(1)v(1,t)=0$ by\eqref{chi1}. Moreover, $u_x=\chi_xv+\chi v_x$, so that $x^\alpha u_x=x^\alpha \chi_x v+\chi x^\alpha v_x$. Using assumption \eqref{hypUC} and property \eqref{chi0} for $\chi$, one gets that $(x^\alpha u_x)(0,t)=0$ for a.e. $t \in (0,T)$. Also, using property \eqref{chi1} for $\chi$, one has $(x^\alpha u_x)(1,t)=0$ for a.e. $t \in (0,T)$.
Thus, we are in a position to  apply Corollary~\ref{cor1} to $u$.  
We obtain
\[
\int_Q s^3l^3u^2e^{2s\phi}+\int_Q slx^{2\alpha+\beta-2}u_x^{2}e^{2s\phi} \leq 
C \int_Q |Lu|^2e^{2s\phi}.
\]
\noindent Replacing $Lu$ by the expression in \eqref{Lu}, we immediately deduce that there exists $C=C(T,\alpha,\beta)>0$ such that
\begin{multline}
\int_Q s^3l^3u^2e^{2s\phi}+\int_Q slx^{2\alpha+\beta-2}u_x^{2}e^{2s\phi}
\\
\leq C\left(\int_Q (|\chi^{\prime\prime}|^2x^{2\alpha}
+|\chi^\prime|^2\alpha^2x^{2\alpha-2})v^2e^{2s\phi} 
+ \int_Q |\chi^\prime|^2x^{2\alpha}v_x^2e^{2s\phi}\right).
\label{estimUC}
\end{multline}
\noindent First of all, using \eqref{chi0} and \eqref{chi1},
\begin{equation}
\int_Q |\chi^{\prime\prime}|^2x^{2\alpha}v^2e^{2s\phi} \leq \int_{\delta^{1/\beta}}^{\eta^{1/\beta}}\int_{0}^{T} |\chi^{\prime\prime}|^2 v^2e^{2s\phi}\,.
\label{estimUC1}
\end{equation}
\noindent 
As for the second term, we have
\[
\int_Q |\chi^\prime|^2\alpha^2x^{2\alpha-2}v^2e^{2s\phi}=\int_{\delta^{1/\beta}}^{\eta^{1/\beta}}\int_{0}^{T}|\chi^\prime|^2\alpha^2x^{2\alpha-2}v^2e^{2s\phi}
\]
\noindent because of \eqref{chi1}. Then,
\begin{equation}
\int_Q |\chi^\prime|^2\alpha^2x^{2\alpha-2}v^2e^{2s\phi} \leq \int_{\delta^{1/\beta}}^{\eta^{1/\beta}}\int_{0}^{T} \eta^\frac{2\alpha-2}{\beta}\alpha^2 |\chi^\prime|^2v^2e^{2s\phi}.
\label{estimUC2}
\end{equation}
\noindent 
Eventually, the last term satisfies the bound
\begin{equation}
\int_Q |\chi^\prime|^2x^{2\alpha}v_x^2e^{2s\phi} \leq \int_{\delta^{1/\beta}}^{\eta^{1/\beta}}\int_{0}^{T} |\chi^\prime|^2x^\alpha v_x^2e^{2s\phi}
\label{estimUC3}
\end{equation}
\noindent 
since $0\leq x\leq 1$. Coming back to \eqref{estimUC} and using \eqref{estimUC1}, \eqref{estimUC2} and \eqref{estimUC3}, we conclude that there exists a constant $C=C(T,\alpha,\beta,\delta,\eta)>0$ such that

\[
\int_Q s^3l^3u^2e^{2s\phi}+\int_Q slx^{2\alpha+\beta-2}u_x^{2}e^{2s\phi} \leq 
C \int_{\delta^{1/\beta}}^{\eta^{1/\beta}}\int_{0}^{T} (|\chi^{\prime\prime}|^2+|\chi^\prime|^2)(v^2+x^\alpha v_x^2)e^{2s\phi}.
\]
\noindent 
Therefore, for some constant $C=C(T,\alpha,\beta,\delta,\eta)>0$,
\[
\int_Q s^3l^3u^2e^{2s\phi}+\int_Q slx^{2\alpha+\beta-2}u_x^{2}e^{2s\phi} \leq 
C \int_{\delta^{1/\beta}}^{\eta^{1/\beta}}\int_{0}^{T} (v^2+x^\alpha v_x^2)e^{2s\phi}\,.
\]
\noindent 
Hence,
\begin{equation}
\int_Q s^3l^3u^2e^{2s\phi} \leq C \int_{\delta^{1/\beta}}^{\eta^{1/\beta}}\int_{0}^{T} (v^2+x^\alpha v_x^2)e^{2s\phi}.
\label{estimUCend}
\end{equation}
\noindent 
Our goal is to estimate the weight $e^{2s\phi}$ from above 
in order to simplify the right-hand side of \eqref{estimUCend}. 
First note that, for all $t\in (0,T)$, 
$l(t)\geq l(\frac{T}{2})=\frac{4}{T^2}$. Also, 
since $p$ is negative and decreasing, for 
all $(x,t) \in (\delta^{1/\beta},\eta^{1/\beta})\times (0,T)$, 
\[
2sp(x)l(t)\leq \frac{8sp(x)}{T^2}\leq \frac{8sp(\delta^{1/\beta})}{T^2}.
\]
\noindent Then, 
\begin{equation}
\int_{\delta^{1/\beta}}^{\eta^{1/\beta}}\int_{0}^{T} (v^2+x^\alpha v_x^2)e^{2s\phi} \leq 
\exp\Big(\frac{8sp(\delta^{1/\beta})}{T^2}\Big)
\|v\|_{L^2(0,T;H^{1}_{a}(0,1))}^2.
\label{estimUCend1}
\end{equation}
\noindent Now, we want to estimate $e^{2s\phi}$ from below, so that we may simplify the left-hand side of \eqref{estimUCend}. We set 
\[
Q_0:=\Big\{(x,t)\in Q~:~p(x)>-\frac{\delta}{3} ,\quad \frac{T}{4}<t<\frac{3T}{4}\Big\}.
\]
\noindent First, since $l(t)\geq \frac{4}{T^2}$ for all $t\in (0,T)$, we have
\[
\int_Q s^3l^3u^2e^{2s\phi} \geq \int_Q s^3\Big(\frac{4}{T^2}\Big)^3u^2e^{2s\phi} \geq \int_{Q_{0}} s^3\Big(\frac{4}{T^2}\Big)^3u^2e^{2s\phi}\,.
\]
\noindent Moreover, $l(t)\leq \frac{16}{3T^2}$ for all $\frac{T}{4}<t<\frac{3T}{4}$. So,  for all $(x,t)\in Q_0$ one has 
\[
2sp(x)l(t)\geq s\frac{32}{3T^2}p(x) \geq \frac{4}{3}\frac{8sp((\frac{\delta}{3})^{1/\beta})}{T^2}.
\]
\noindent 
Consequently,
\[
\begin{split}
\int_{Q_{0}} s^3\Big(\frac{4}{T^2}\Big)^3u^2e^{2s\phi} & \geq s^3 
\exp\Big(\frac{4}{3}\frac{8sp((\frac{\delta}{3})^{1/\beta})}{T^2}\Big)
\int_{Q_{0}}\Big(\frac{4}{T^2}\Big)^3u^2, \\
& = s^3 
\exp\Big(\frac{4}{3}\frac{8sp((\frac{\delta}{3})^{1/\beta})}{T^2}\Big)
\int_{Q_{0}}\Big(\frac{4}{T^2}\Big)^3\chi^2v^2.
\end{split}
\]
\noindent Note that $p(x)>-\frac{\delta}{3} \Longleftrightarrow x \in (0,(\frac{\delta}{3})^{1/\beta})$. So, on account of \eqref{chi0},
\[
s^3 \exp\Big(\frac{4}{3}\frac{8sp((\frac{\delta}{3})^{1/\beta})}{T^2}\Big) \int_{Q_{0}}\Big(\frac{4}{T^2}\Big)^3\chi^2v^2=s^3 \exp\Big(\frac{4}{3}\frac{8sp((\frac{\delta}{3})^{1/\beta})}{T^2}\Big) \int_{Q_{0}}\Big(\frac{4}{T^2}\Big)^3v^2.
\]
\noindent Finally, 
\begin{equation}
\int_Q s^3l^3u^2e^{2s\phi} \geq s^3 \exp\Big(\frac{4}{3}\frac{8sp((\frac{\delta}{3})^{1/\beta})}{T^2}\Big) \int_{Q_{0}}\Big(\frac{4}{T^2}\Big)^3v^2.
\label{estimUCend2}
\end{equation}
\noindent 
Coming back to \eqref{estimUCend}, and using \eqref{estimUCend1} and \eqref{estimUCend2} we have
\begin{multline*}
s^3\Big(\frac{4}{T^2}\Big)^3\|v\|_{L^2(Q_0)}^2\exp\Big(\frac{4}{3}\frac{8sp((\frac{\delta}{3})^{1/\beta})}{T^2}\Big)
\\
 \leq C(T,\alpha,\beta,\delta) 
\exp\Big(\frac{8sp((\frac{\delta}{3})^{1/\beta})}{T^2}\Big)
{T^2}\|v\|_{L^2(0,T;H^{1}_{\alpha}(0,1))}^2\,,
\end{multline*}
\noindent 
from which we immediately deduce that
\[
\|v\|_{L^2(Q_0)}^2 \leq C(T,\alpha,\beta,\delta) \|v\|_{L^2(0,T;H^{1}_{\alpha}(0,1))}^2 \frac{1}{s^3} 
\exp\Big(
\frac{8s}{T^2}\Big[p(\delta^{1/\beta})-\frac{4}{3}p\big((\delta/3\big)^{1/\beta})\Big]
\Big)\,.
\]
\noindent 
Now, $p(\delta^{1/\beta})
-\frac{4}{3}p((\delta/3)^{1/\beta})
=-\delta+\frac{4}{3}\frac{\delta}{3}=-\frac{5\delta}{9}$. 
Passing to the limit when $s\to\infty$, 
we have that $\|v\|_{L^2(Q_0)}^2=0$. In conclusion, 
\[
v \equiv 0 \quad \text{ in }\quad
\Big(0,\Big(\frac{\delta}{3}\Big)^{1/\beta}\Big) \times \Big(\frac{T}{4},\frac{3T}{4}\Big).
\] 

To complete the proof, observe that the classical unique continuation for parabolic equations
implies that $v\equiv 0$ in $(0,1)\times (T/4,3T/4)$. Equivalently, $e^{(T-t)A}v(T) = 0$ for all $t\in(T/4,3T/4)$, where $e^{tA}$ is the semigroup generated by $A$. Since $e^{tA}$ is analytic for $t>0$, this implies that $v\equiv 0$ in $(0,1)\times (0,T)$.
\end{proof}

\section{From unique continuation to approximate controllability}
\label{se:ac}
Let $0<\alpha <1$ and fix $T>0$. We are interested in the following initial-boundary value problem
\begin{equation}
\left\{ 
\begin{array}{lc}
u_{t}-(x^\alpha u_{x})_{x}=0 &\quad \left( x,t\right) \in Q= \left(0,1\right) \times \left(0,T\right) \\ 
u\left( 0,t\right) =g(t) & \quad t\in \left( 0,T\right) \\ 
u\left(1,t\right)=0 &\quad t\in \left(0,T\right) \\ 
u\left( x,0\right) =u_{0}(x) & \quad x\in \left( 0,1\right).
\end{array}
\right.  
\label{Pbm1}
\end{equation}
We aim at proving approximate controllability at time $T$ for the above equation, which amounts to showing that for any final state $u_T$ and any arbitrarily small neighbourhood $\mathcal V$ of $u_T$, there exists a control $g$ driving the solution of \eqref{Pbm1} to $\mathcal V$ at time T. 

Boundary control problems can be recast in abstract form in a standard way, see, e.g., \cite{BPDM}. Here, we follow a simpler method working directly on the parabolic  problem, where the boundary control is reduced to a suitable forcing term.
We begin by discussing the existence and uniqueness of solutions for \eqref{Pbm1}.

\subsection{Well-posedness of \eqref{Pbm1}}

\begin{theorem}
For all $u_0\in H_{\alpha,0}^{1}\left( 0,1\right)$ and all $g \in H^{1}_{0}(0,T)$, problem \eqref{Pbm1} has a unique mild solution $u \in L^2(0,T;H_{\alpha}^{1}\left( 0,1\right) \cap C([0,1]; L^2(0,1))$. Moreover,
\begin{equation}
\sup_{t\in [0,T]} \left\|u(t)\right\|_{L^2(0,1)}^2+\left\|x^{\alpha/2}u_x\right\|_{L^2(0,T;L^2(0,1))}^2 \leq C(T)( \left\|g\right\|_{H^{1}_{0}(0,T)}^2+\left\|u_0\right\|_{L^2(0,1)}^2).
\label{estimsol}
\end{equation}
\noindent Furthermore, $(x^\alpha u_x)_x \in L^2(0,T;L^2(0,1))$ and   \eqref{Pbm1} is satisfied almost everywhere.
\label{theo3}
\end{theorem}
\begin{proof}
 Let $u_0\in H_{\alpha,0}^{1}\left( 0,1\right)$ and $g \in H^{1}_{0}(0,T)$. Let us introduce the initial-boundary  value problem with homogeneous boundary conditions
\begin{equation}
\left\{ 
\begin{array}{lc}
y_{t}-(x^\alpha y_{x})_{x}=-(1-x^{1-\alpha})g_t &\quad \left( x,t\right) \in  Q\\ 
y\left( 0,t\right) =0 & \quad t\in \left( 0,T\right) \\
y\left( 1,t\right)=0 &\quad t\in \left(0,T\right) \\  
y\left( x,0\right) =u_{0}(x) & \quad x\in \left( 0,1\right).
\end{array}
\right.  
\label{Pbm2}
\end{equation}

\noindent Let us first prove the existence of a solution of \eqref{Pbm1}. Using the fact that $A$ is the infinitesimal generator of an analytic semigroup, we know that problem \eqref{Pbm2} has a unique solution $y \in L^2(0,T;D(A)) \cap H^1(0,T;L^2(0,1))$ (see for instance \cite{Cannarsa1,Cannarsa5}). Moreover, multiplying the first equation of \eqref{Pbm2} by $y$ and integrating over $Q$,
\begin{equation}
\sup_{t\in [0,T]} \left\|y(t)\right\|_{L^2(0,1)}^2+\left\|x^{\alpha/2}y_x\right\|_{L^2(0,T;L^2(0,1))}^2 \leq C(T,\alpha)( \left\|g\right\|_{H^{1}_{0}(0,T)}^2+\left\|u_0\right\|_{L^2(0,1)}^2).
\label{estimy}
\end{equation}
\noindent Set, for a.e. $(x,t) \in Q$, 
\begin{equation}
u(x,t):=y(x,t)+(1-x^{1-\alpha})g(t).
\label{expressionu}
\end{equation}

\noindent Then, $u \in H^1(0,T;L^2(0,1))\cap L^2(0,T;H_{\alpha}^{1}\left( 0,1\right))$ and, as we observed in Example~\ref{lem1}, 
$(x^\alpha u_x)_x=(x^\alpha y_x)_x \in L^2(0,T;L^2(0,1))$. Moreover,
\begin{multline*}
u_t(x,t)=y_t(x,t)+(1-x^{1-\alpha})g_t(t)
\\
=(x^\alpha y_x)_x(x,t)-(1-x^{1-\alpha})g_t(t)+(1-x^{1-\alpha})g_t(t)
\\=(x^\alpha y_x)_x(x,t) =(x^\alpha u_x)_x(x,t)\,.
\end{multline*}
\noindent  for a.e. $(x,t) \in Q$. Since $u \in L^2(0,T;H_{\alpha}^{1}\left( 0,1\right))$, for a.e. $t \in (0,T)$, $u(0,t)$ and $u(1,t)$ exist. Therefore, using \eqref{expressionu}, $u(0,t)=g(t)$ and $u(1,t)=0$. Also, for a.e. $x \in (0,1)$, $u(x,0)=y(x,0)=u_0(x)$ since $g \in H^{1}_{0}(0,T)$. Consequently, $u$ is a mild solution of \eqref{Pbm1} satisfying $(x^\alpha u_x)_x\in L^2(0,T;L^2(0,1))$ and $u \in H^1(0,T;L^2(0,1))$. Finally,  estimate \eqref{estimsol} follows from \eqref{estimy} and \eqref{expressionu}.

Next, let us prove uniqueness. Let $u_1$ and $u_2$ be two solutions of \eqref{Pbm1}. Then, the difference $w:=u_1-u_2$ is a solution of \eqref{Pbm2}, with $g\equiv 0$ and $u_0\equiv 0$. Because of the uniqueness property of problem \eqref{Pbm2}, $w \equiv 0$.
\end{proof}

\subsection{Approximate controllability}

Our goal is now to show the following theorem.

\begin{theorem}
Let $u_0 \in H_{\alpha,0}^{1}\left( 0,1\right)$. For all $u_T \in L^2(0,1)$ and all $\epsilon >0$ there exists $g \in H^{1}_{0}(0,T)$ such that the solution $u_g$ of problem \eqref{Pbm1} satisfies 
\[
\left\|u_g(T)-u_T\right\|_{L^2(0,1)} \leq \epsilon\,.
\]
\label{theo2}
\end{theorem}

We start the proof with a  lemma.
\begin{lemma}
If the conclusion of Theorem~\ref{theo2} is true for $u_0 \equiv 0$, then it is true for any $u_0 \in H_{\alpha,0}^{1}\left( 0,1\right)$.
\label{lem2}
\end{lemma}
\begin{proof}
Let $u_0 \in H_{\alpha,0}^{1}\left( 0,1\right)$ and $u_T \in L^2(0,1)$. Let $\epsilon >0$. Let us introduce $\hat{u}$ the (mild) solution of 
\\
\[
\left\{ 
\begin{array}{lc}
\hat{u}_{t}-(x^\alpha \hat{u}_{x})_{x}=0 &\quad \left( x,t\right) \in  Q\\ 
\hat{u}\left( 0,t\right) =0 & \quad t\in \left( 0,T\right) \\ 
\hat{u}\left( 1,t\right)=0 &\quad t\in \left(0,T\right) \\ 
\hat{u}\left( x,0\right) =u_{0}(x) & \quad x\in \left( 0,1\right).
\end{array}
\right.  
\]
\\
\noindent 
Then, $\hat{u}(T) \in L^2(0,1)$. Therefore, using the assumption of Lemma \ref{lem2}, there exists $g \in H^{1}_{0}(0,T)$ such that the solution $\upsilon_g$ of\\ 
\[
\left\{ 
\begin{array}{lc}
\upsilon_{t}-(x^\alpha \upsilon_{x})_{x}=0 &\quad \left( x,t\right) \in  Q\\ 
\upsilon\left( 0,t\right) =g(t) & \quad t\in \left( 0,T\right) \\
\upsilon\left( 1,t\right)=0 &\quad t\in \left(0,T\right) \\  
\upsilon\left( x,0\right) =0 & \quad x\in \left( 0,1\right).
\end{array}
\right.  
\]
\\
\noindent satisfies 
\[
\left\|\upsilon_g(T)-(u_T-\hat{u}(T))\right\|_{L^2(0,1)} \leq \epsilon.
\]
\noindent Yet, one can easily see that $u_g(T)=\upsilon_g(T)+\hat{u}(T)$, so that the proof of Lemma~\ref{lem2} is achieved. 
\end{proof}
We now assume that $u_0\equiv 0$. 
\begin{lemma}
For all $g \in H^{1}_{0}(0,T)$, for all $v \in L^2(0,1)$,
\begin{equation}
\left(u_g\left(T\right),v\right)_{L^2(0,1)}=\int _{0}^{T}(x^\alpha \hat{v}_x)(0,t)g(t)dt,
\label{egortho}
\end{equation}
\noindent where $\hat{v} \in C([0,T];L^2(0,1))\cap L^2(0,T;H^{1}_{\alpha,0})$ is the solution of
\begin{equation}
\left\{ 
\begin{array}{lc}
\hat{v}_{t}+(x^\alpha \hat{v}_{x})_{x}=0 &\quad \left( x,t\right) \in  Q\\ 
\hat{v}\left(t,0\right) =0 & \quad t\in \left( 0,T\right) \\
\hat{v}\left(t,1\right)=0 &\quad t\in \left(0,T\right) \\  
\hat{v}\left(T,x\right) =v(x) & \quad x\in \left( 0,1\right).
\end{array}
\right.  
\label{Pbm3}
\end{equation}
\label{lem3}
\end{lemma}
\begin{proof}
  Let us multiply by $\hat{v}$ the equation satisfied by $u_g$. Then, integrating by parts with respect to the space variable, one has, for almost all $t \in (0,T)$,
\begin{equation}
\left(u_{g,t}(t),\hat{v}(t)\right)_{L^2(0,1)}+\int_{0}^{1}x^{\alpha/2}u_{g,x}(t)x^{\alpha/2}\hat{v}_x(t)dx=0.
\label{estimorthog1}
\end{equation}

\noindent Moreover, for all $\eta>0$, $\hat{v} \in L^2(0,T-\eta;D(A))\cap H^1(0,T-\eta;L^2(0,1))$. We multiply by $u_g$ the equation satisfied by $\hat{v}$ on $(0,T-\eta)$. After a standard integration by parts with respect to the space variable, one has, for a.e. $t \in (0,T-\eta)$,
\begin{equation}
\left(u_{g}(t),\hat{v}_t(t)\right)_{L^2(0,1)}-\int_{0}^{1}x^{\alpha/2}u_{g,x}(t)x^{\alpha/2}\hat{v}_x(t)dx=(x^{\alpha}\hat{v})_x(0,t)g(t).
\label{estimorthog2}
\end{equation}

\noindent Adding \eqref{estimorthog1} and \eqref{estimorthog2}, one gets, for a.e. $t \in (0,T-\eta)$,
\[
\frac{d}{dt}\left(u_{g}(t),\hat{v}(t)\right)_{L^2(0,1)}=(x^{\alpha}\hat{v})_x(0,t)g(t).
\]
\noindent Now, integrating over $(0,T-\eta)$ and recalling that $u_g(0)=u_0=0$, one obtains 
\begin{equation}
\left(u_{g}(T-\eta),\hat{v}(T-\eta)\right)_{L^2(0,1)}=\int_{0}^{T-\eta}(x^{\alpha}\hat{v})_x(0,t)g(t)dt.
\label{estimorthog3}
\end{equation}
\noindent Since $u_g \in C([0,T];L^2(0,1))$, $\hat{v} \in C([0,T];L^2(0,1))$ and $\hat{v}(T)=v$, one gets
\[
\left(u_g\left(T\right),v\right)_{L^2(0,1)}=\int _{0}^{T}(x^\alpha \hat{v}_x)(0,t)g(t)dt,
\]
\noindent passing to the limit as $\eta\downarrow 0$.
\end{proof}
Finally, define the control operator $B$ by
\begin{equation*}
B:  H^{1}_{0}(0,T)  \longrightarrow  L^2(0,1)\,,\quad
B:g  \longmapsto  u_g(T)
\end{equation*}
According to \eqref{estimsol}, $B \in \mathcal L(H^{1}_{0}(0,T),L^2(0,1))$. Then, problem \eqref{Pbm1} is approximately controllable if and only if the range of $B$ is dense in $L^2(0,1)$. This is equivalent to the fact that the orthogonal of $\mathcal R (B)$ is reduced to $\{0\}$.

\begin{lemma}
If $v \in \mathcal R (B)^{\bot}$, then $(x^\alpha \hat{v}_x)(.,0)\equiv 0$.
\label{lem4}
\end{lemma}
\begin{proof}
Take $v \in \mathcal R (B)^{\bot}$. According to \eqref{egortho}, for all $g \in H^{1}_{0}(0,T)$, 
\[
\int _{0}^{T}(x^\alpha \hat{v}_x)(0,t)g(t)dt=0.
\]
\noindent Even if $t\longmapsto (x^\alpha \hat{v}_x)(0,t)$ is not a-priori in $L^2(0,T)$, we can conclude that $(x^\alpha \hat{v}_x)(.,0)\equiv 0$. Indeed, take $\eta>0$. Take $g \in \mathcal D(0,T-\eta)$ and set $g\equiv 0$ on $(T-\eta,T)$. Then $g \in H^{1}_{0}(0,T)$ and 
\[
0=\int _{0}^{T}(x^\alpha \hat{v}_x)(0,t)g(t)dt=\int _{0}^{T-\eta}(x^\alpha \hat{v}_x)(0,t)g(t)dt.
\]
\noindent Yet, $t\longmapsto (x^\alpha \hat{v}_x)(0,t) \in L^2(0,T-\eta)$, so that, by density, for all $g \in L^2(0,T-\eta)$,
\[
\int _{0}^{T-\eta}(x^\alpha \hat{v}_x)(0,t)g(t)dt=0\,.
\]
Therefore, $(x^\alpha \hat{v}_x)(\cdot,0)\equiv 0$ on $(0,T-\eta)$ for all $\eta>0$.  
\end{proof}
In order to complete the proof of Theorem \ref{theo2}, we just need to apply our unique continuation result: since the solution $\hat{v}$ of \eqref{Pbm3} satisfies $(x^\alpha \hat{v}_x)(.,0)\equiv 0$ on $(0,T)$, we have that $\hat{v}(T)=v=0$.

\begin{remark}
Theorem~\ref{theo2} yields the approximate controllability in $L^2(0,1)$ of problem \eqref{Pbm1}, as is easily seen arguing as follows. Let  $T>0$, let $\epsilon>0$ and let $u_0, u_T\in L^2(0,1)$. Set $u_1=e^{TA/2}u_0$ and observe that, since the semigroup is analytic, $u_1\in H^1_{\alpha,0}(0,1)$. Therefore, owing to Theorem~\ref{theo2}, there exists $g_1\in H^1_0(T/2,T)$ such that the solution of the problem
\begin{equation*}
\left\{ 
\begin{array}{lc}
u_{t}-(x^\alpha u_{x})_{x}=0 &\quad \left( x,t\right) \in  \left(0,1\right) \times \left(T/2,T\right) \\ 
u\left( 0,t\right) =g_1(t) & \quad t\in \left( T/2,T\right) \\ 
u\left(1,t\right)=0 &\quad t\in \left(T/2,T\right) \\ 
u\left( x,T/2\right) =u_{1}(x) & \quad x\in \left( 0,1\right).
\end{array}
\right.  
\end{equation*}
satisfies $\left\|u(T)-u_T\right\|_{L^2(0,1)} \leq \epsilon$. Thus, a boundary control $g$ for  \eqref{Pbm1} which steers the system into an $\epsilon$-neighborhood of $u_T$ is given by
\begin{equation*}
g(t)=
\begin{cases}
0
&t\in[0,T/2)
\\
g_1(t) & t\in [T/2,T]\,.
\end{cases}
\end{equation*}
\end{remark}

\section*{Acknowledgement}
We would like to thank the referee who caught numerous errors in an earlier draft of the paper.


\begin{thebibliography}{10}
\markboth{Taylor \& Francis and I.T. Consultant}{Applicable Analysis}
\bibitem[1]{Alabau} 
F. Alabau-Boussouira, P. Cannarsa, and G. Fragnelli, {\em Carleman estimates for degenerate parabolic operators with application to null 
controllability},  J. evol. equ. 6 (2006), vol. 2, 161--204.

\bibitem[2]{BPDM} A. Bensoussan, G. Da Prato, M. C. Delfour, and S. K. Mitter, {\em Representation and Control of Infinite Dimensional Systems}, Birkh\"auser, 2007.

\bibitem[3]{Bu}
J.-M. Buchot and J.-P. Raymond, {\em A linearized model for 
boundary layer equations}, 
in Optimal Control of Complex Structures (Oberwolfach, 2000),
Internat. Ser. Numer. Math. 139, BirkhNauser, Basel, 2002, 
31--42.

\bibitem[4]{non-div0} 
P. Cannarsa, G. Fragnelli, and D. Rocchetti, {\em Null controllability of degenerate parabolic operators with drift},  Netw. Heterog. Media  
2  (2007),  no. 4, 695--715.

\bibitem[5]{non-div} 
P. Cannarsa, G. Fragnelli, and D. Rocchetti, {\em Controllability results for a class of one-dimensional degenerate parabolic problems 
in nondivergence form},  J. evol. equ. 8 (2008), 583--616.

\bibitem[6]{Cannarsa1} P. Cannarsa,P. Martinez, and J. Vancostenoble, {\em Null controllability of degenerate heat equations}, 
Adv. Differential Equations 10 (2005), no. 2, 153--190.

\bibitem[7]{Cannarsa}P. Cannarsa, P. Martinez, and 
J. Vancostenoble, {\em Carleman estimates for a class of degenerate parabolic operators}, SIAM J. Control Optim. 47 (2008), no. 1, 1--19.

\bibitem[8]{CRV} 
P. Cannarsa, D. Rocchetti, and J. Vancostenoble, 
{\em Generation of analytic semi-groups in $L^2$ for 
a class of second order degenerate elliptic operators},
 Control Cybernet. 37 (2008),  831--878.

\bibitem[9]{Cannarsa5}P. Cannarsa, J. Tort and 
M. Yamamoto, {\em Determination of a source term 
in a degenerate parabolic equation}, 
Inverse Problems 26 (2010), 20 pp.

\bibitem[10]{FuIm} A. V. Fursikov and O. Y. Imanuvilov, 
{\em Controllability of Evolution Equations}, 
Lecture Notes Ser. 34, Seoul National University, 
Seoul, Korea, 1996.

\bibitem[11]{Im} O. Y. Imanuvilov, 
{\em Boundary controllability of parabolic equations}, Mat.
Sb. 186 (1995), 109--132.

%
%
%
%
%
%
%
%
%
%
\end{thebibliography}
\end{document}